\documentclass[a4paper,11pt]{article}
\usepackage[english]{babel}
\usepackage[utf8]{inputenc}
\usepackage{paralist,url,verbatim,anysize}
\usepackage{amscd}
\usepackage{mathrsfs}
\usepackage{microtype}
\usepackage{manfnt}
\usepackage{nicefrac}

\usepackage[nodisplayskipstretch]{setspace}
\usepackage{fancyhdr}
\fancypagestyle{plain}{%
\fancyhf{} % clear all header and footer fields
\fancyfoot[C]{\fontsize{12pt}{12pt}\selectfont\thepage} % except the center

}
\usepackage[e]{esvect}
\usepackage{amsmath,amsthm,amssymb,amsfonts}
\usepackage{graphicx, graphics}
\usepackage[bookmarksopen=true]{hyperref}
\hypersetup{colorlinks=true, citecolor=blue}
\usepackage{epstopdf}
\usepackage[mathlines, pagewise]{lineno}
\usepackage{bbm}
\usepackage[font=small,labelfont=bf]{caption}
\usepackage{subcaption}
\theoremstyle{plain}
\newtheorem{theorem}{\bf Theorem}[subsection]
\newtheorem{conjecture}[theorem]{Conjecture}
\newtheorem{cor}[theorem]{Corollary}
\newtheorem{lemma}[theorem]{Lemma}

\newtheorem{problem}[theorem]{Problem}

\newtheorem{thmnonumber}{\bf Main Theorem}

\theoremstyle{definition}

\newtheorem{definition}[theorem]{Definition}

\DeclareMathAlphabet{\mathpzc}{OT1}{pzc}{m}{it}
\newenvironment{frcseries}{\fontfamily{frc}\selectfont}{}

\newcommand{\NE}{\operatorname{NE} }
\newcommand{\TT}{\mathrm{T}}
\newcommand{\F}{\mathrm{F}}
\newcommand{\rint}{\mathrm{relint}}
\newcommand{\RN}{\mathrm{N}}

\newcommand{\R}{\mathbb{R}}
\newcommand{\RS}{\mathrm{R}\,}

\newcommand{\Lk}{\mathrm{Lk}\, }
\newcommand{\LLk}{\mathrm{LLk}^{\,\nu} }
\newcommand{\SLk}{\mathrm{SLk}^{\,\nu} }

\newcommand{\St}{\mathrm{St}\, }

\newcommand{\intx}{\mathrm{int}\,}
\newcommand{\sd}{\mathrm{sd}\, }

\newcommand{\cm}[1]{}
\newcommand{\io}{(\operatorname{A})}
\newcommand{\iit}{(\operatorname{B})}
\newcommand{\iii}{(\operatorname{C})}

\begin{document}

\author{
Karim Adiprasito
\\
\small Einstein Institute for Mathematics\\ \small Hebrew University of Jerusalem\\
\small 91904 Jerusalem, Israel\\
\small \url{adiprasito@math.huji.ac.il}
\and
\and 
Bruno Benedetti \\
\small Department of Mathematics\\ \small University of Miami\\
\small 33146 Coral Gables, Florida\\
\small \url{bruno@math.miami.edu}}

\date{\small September 22, 2015}
\title{Barycentric subdivisions of convex complexes are collapsible}
\maketitle
\bfseries

\mdseries

%\enlargethispage{4mm}

\begin{abstract}
A classical  question in PL topology, asked among others by Hudson, Lickorish, and Kirby, is whether every linear subdivision of the $d$-simplex is simplicially collapsible. The answer is known to be positive for $d \le 3$. We solve the problem up to one subdivision, by proving that any linear subdivision of any polytope is simplicially collapsible after at most one barycentric subdivision. 
%This solves up to one derived subdivision a classical question by Hudson and Lickorish.
Furthermore, we prove that any linear subdivision of any star-shaped polyhedron in $\R^d$ is simplicially collapsible after $d-2$ derived subdivisions at most. This presents progress on an old question by Goodrick.
\end{abstract}

%\tableofcontents

\section{Introduction}

Collapsibility is a combinatorial version of the notion of contractibility, introduced in 1939 by Whitehead. All triangulation of the $2$-dimensional ball are collapsible. In contrast, Bing and Goodrick showed how to construct non-collapsible triangulations of the $d$-ball for each $d\ge 3$ \cite{BING} \cite[Corollary 4.25]{Benedetti-DMT4MWB}. (See  also \cite{BL13} for an explicit example with $15$ vertices.) 

In Bing's $3$-dimensional examples, the obstruction to collapsibility is the presence of subcomplexes with few facets that are isotopic to knots.  This is also an obstruction to admitting a convex geometric realization: For example, if a $3$-ball contains a knot realized as subcomplex with $\le 5$ edges, then the $3$-ball cannot be embedded in $\R^3$ (since the stick number of the trefoil is $6$). In particular, such $3$-ball cannot have any convex geometric realization in any $\R^k$. 

This lead Bing to ask whether simplicial subdivisions of convex 3-polytopes are collapsible. In 1967, Chillingworth answered Bing's question in the affirmative.  Chillingworth's proof is based on an elementary induction. Consider the highest vertex $v$ of the complex (according to some generic linear functional). The link of $v$ is necessarily \emph{planar}, and all planar $2$-complexes are collapsible. In particular, the star of $v$ collapses to the link of $v$. This implies that the the complex $C$ collapses to $C-v$, the subcomplex of $C$ consisting of all faces that do not contain $v$. 
%,  \emph{nonevasive} (a stronger property than collapsibility). Intuitively, this implies that you can collapse the star of the top vertex 

Unfortunately, Chillingworth's argument is specific to dimension $3$, because vertex links in a $4$-ball are no longer planar. In fact, whether all convex $d$-balls are collapsible represents a long-standing open problem, where little progress has been made since the Sixties. The problem has appeared  in the literature in at least three different versions:

\begin{conjecture}[{Lickorish's Conjecture, cf.~Kirby~\cite[Problem 5.5 (A)]{Kirby}}]
Let $C$ be a simplicial complex. If $C$ is a subdivision of the simplex, then $C$ is collapsible.
\end{conjecture}

\begin{conjecture}[{Goodrick's Conjecture, cf.~Kirby~\cite[Problem 5.5 (B)]{Kirby}}]
Let $C$ be a simplicial complex.  If (the underlying space of) $C$ is star-shaped, then $C$ is collapsible.
\end{conjecture}

\begin{problem}[{Hudson's Problem~\cite[Sec.~2, p.~44]{Hudson}}] Let $C$ be a simplicial complex. 
 If $C$ collapses onto some subcomplex $C'$, does every simplicial subdivision $D$ of $C$ collapse to the restriction $D'$ of $D$ to the underlying space of~$C'$?
\end{problem}

\begin{comment}
\begin{problem} Let $C$ be an arbitrary simplicial complex.
\begin{compactenum}[\rm (1)]
\item \emph{(Lickorish's Conjecture)~\cite[Prb.\ 5.5 (A)]{Kirby}} If $C$ is a subdivision of the simplex, then $C$ is collapsible.
\item \emph{(Goodrick's Conjecture)~\cite[Prb.~5.5 (B)]{Kirby}} If (the underlying space of) $C$ is star-shaped, then $C$ is collapsible.
\item \emph{(Hudson's Problem)~\cite[Sec.~2, p.~44]{Hudson}} If $C$ collapses onto some subcomplex $C'$, does every simplicial subdivision $D$ of $C$ collapse to the restriction $D'$ of $D$ to the underlying space of~$C'$?
\end{compactenum}
\end{problem}
\end{comment}

%Here we show that Chillingworth's result can indeed be generalized to higher dimensions, although what we achieve is a slightly weaker form than the one generally conjectured.
In this paper, we show that all three problems above can be solved if we are allowed to modify the complex $C$ by performing a bounded number of barycentric subdivisions. Our bounds are universal, i.e. they do not depend on the complex chosen.

The new idea is to ``resuscitate'' Chillingworth's inductive method by expanding the problem into spherical geometry. In fact, if $v$ is the top vertex of a (geometric) simplicial complex $C \subset \mathbb{R}^d$, the vertex link of $v$ has a natural geometric realization as \emph{spherical} simplicial complex, obtained by intersecting $C$ with a small $(d-1)$-sphere centered at~$v$. %Thus it is natural to view the vertex link of $v$ as a spherical simplicial complex. 

Our trick is to find a special subdivision $S$ of the complex $C$ in which the link of the top vertex is a \emph{(geodesically) convex} subset of the sphere. It turns out that a subdivision combinatorially equivalent to the barycentric subdivision does the trick. So we aim for a stronger statement, namely, that both convex $d$-complexes in $\R^d$ and convex spherical $d$-complexes in $S^d$ become collapsible after one barycentric subdivision. This way we can proceed by induction on the dimension: The inductive assumption will tells us that the subdivided link of the top vertex is collapsible. 

%So essentially, the collapsibility of convex $d$-dimensional complexes in Euclidean space relies on the collapsibility of 

%With this new tool, we try Chillingworth's theorem, which simultaneously claims the nonevasiveness of convex $d$-dimensional complexes in Euclidean space \emph{and} of convex (spherical) $d$-dimensional complexes in the sphere. To prove such statement, we resuscitate the inductive mechanism in Chillingworth's proof: the (subdivision of the) top vertex link is necessarily nonevasive, this time not because of planarity!, but because it is convex in a $(d-1)$-sphere -- and we can apply induction. 

After a few technicalities, this idea takes us to the following sequence of results. 

\begin{thmnonumber}\label{mainthm:Convex} Let $C$ be an arbitrary simplicial complex in $\mathbb{R}^d$.
\begin{compactenum}[\rm (1)]
\item \emph{(Theorem~\ref{thm:liccon})} If the underlying space of $C$ in $\mathbb{R}^d$ is convex, then the (first) derived subdivision of $C$ is collapsible.
\item \emph{(Theorem~\ref{thm:ConvexEndo})} If the underlying space of $C$ in $\mathbb{R}^d$ is star-shaped,
then its $(d-2)$-nd derived subdivision is collapsible.
\item \emph{(Theorem~\ref{thm:hudson})}
If $C$ collapses simplicially onto some subcomplex $C'$, then for every simplicial subdivision $D$ of $C$ the simplicial complex $\sd D$ collapses to the restriction of $\sd D$ to the underlying space of $C'$.
\end{compactenum}
\end{thmnonumber}

\section{Preliminaries} \label{sec:Preliminaries}

\subsection{Geometric and intrinsic polytopal complexes}
By $\R^d$ and $S^d$ we denote the euclidean $d$-space and the unit sphere in $\R^{d+1}$, respectively. A \emph{(Euclidean) polytope} in $\R^d$ is the convex hull of finitely many points in $\R^d$. A \emph{spherical polytope} in $S^d$ is the convex hull of a finite number of points that all belong to some open hemisphere of $S^d$. Spherical polytopes are in natural one-to-one correspondence with euclidean polytopes, by taking radial projections. A \emph{geometric polytopal complex} in $\R^d$ (resp.\ in $S^d$) is a finite collection of polytopes in $\R^d$ (resp.~$S^d$) such that the intersection of any two polytopes is a face of both. An \emph{intrinsic polytopal complex} is a collection of polytopes that are attached along isometries of their faces (cf.\ Davis--Moussong~\cite[Sec.\ 2]{DavisMoussong}), so that the intersection of any two polytopes is a face of both. 

Two polytopal complexes $C,\, D$ are \emph{combinatorially equivalent}, denoted by $C\cong D$, if their face posets are isomorphic. Any polytope combinatorially equivalent to the $d$-simplex, or to the regular unit cube $[0,1]^d$, shall simply be called a \emph{$d$-simplex} or a \emph{$d$-cube}, respectively. A polytopal complex is \emph{simplicial}  (resp.~\emph{cubical}) if all its faces are simplices (resp.~cubes). %The set of $k$-dimensional faces of a polytopal complex $C$ is denoted by $\F_k(C)$, and the cardinality of this set is denoted by $f_k(C).$ 

The \emph{underlying space} $|C|$ of a polytopal complex $C$ is the topological space obtained by taking the union of its faces. If two complexes are combinatorially equivalent, their underlying spaces are homeomorphic. We will frequently abuse notation and identify a polytopal complex with its underlying space, as is common in the literature. For instance, we do not distinguish between a polytope and the complex formed by its faces. If $C$ is simplicial, $C$ is sometimes called a \emph{triangulation} of $|C|$ (and of any topological space homeomorphic to $|C|$). If $|C|$ is isometric
to some metric space $M$, then $C$ is called a \emph{geometric triangulation} of~$M$.

A \emph{subdivision} of a polytopal complex $C$ is a polytopal complex $C'$ with the same underlying space of $C$, such that for every face $F'$ of $C'$ there is some face $F$ of $C$ for which $F' \subset F$. Two polytopal complexes $C$ and $D$ are called \emph{PL equivalent} if some subdivision $C'$ of $C$ is combinatorially equivalent to some subdivision $D'$ of $D$. In case $|C|$ is a topological manifold (with or without boundary), we say that $C$ is \emph{PL} (short for Piecewise-Linear) if the star of every face of $C$ is PL equivalent to the simplex of the same dimension.

A \emph{derived subdivision} $\sd   C$ of a polytopal complex $C$ is any subdivision of $C$ obtained by stellarly subdividing at all faces in order of decreasing dimension of the faces of $C$, cf.\ \cite{Hudson}.  An example of a derived subdivision is the \emph{barycentric subdivision}, which uses as vertices the barycenters of all faces of $C$.
 
If $C$ is a polytopal complex, and $A$ is some set, we define the \emph{restriction $\RS(C,A)$ of $C$ to $A$} as the inclusion-maximal subcomplex $D$ of $C$ such that $D$ lies in $A$. The \emph{star} of $\sigma$ in $C$, denoted by $\St(\sigma, C)$, is the minimal subcomplex of $C$ that contains all faces of $C$ containing $\sigma$. The \emph{deletion} $C-D$ of a subcomplex $D$ from $C$ is the subcomplex of $C$ given by $\RS(C,  C{\setminus} \rint{D})$. The \emph{(first) derived neighborhood} $N(D,C)$ of $D$ in $C$ is the simplicial complex
\[N(D,C):=\bigcup_{\sigma\in \sd D} \St(\sigma,\sd C). \] 

Next, we define (a geometric realization of) the \emph{link} with a metric approach. (We took inspiration from Charney \cite{Charney} and Davis--Moussong~\cite[Sec.\ 2.2]{DavisMoussong}.) Let $p$ be any point of a metric space $X$. By $\TT_p X$ we denote the tangent space of $X$ at $p$. Let $\TT^1_p X$ be the restriction of $\TT_p X$ to unit vectors.  If $Y$ is any subspace of $X$, then $\RN_{(p,Y)} X$ denotes the subspace of the tangent space $\TT_p X$ spanned by the vectors orthogonal to $\TT_p Y$. If $p$ is in the interior of $Y$, we define $\RN^1_{(p,Y)} X:= \RN_{(p,Y)} X \cap \TT^1_p Y$.
If $\tau$ is any face of a polytopal complex $C$ containing a nonempty face $\sigma$ of $C$, then the set $\RN^1_{(p,\sigma)} \tau$ of unit tangent vectors in $\RN^1_{(p,\sigma)} |C|$ pointing towards $\tau$ forms a spherical polytope $P_p(\tau)$, isometrically embedded in $\RN^1_{(p,\sigma)} |C|$. The family of all polytopes $P_p(\tau)$ in $\RN^1_{(p,\sigma)} |C|$ obtained for all $\tau \supset \sigma$ forms a polytopal complex, called the \emph{link} of $C$ at $\sigma$; we will denote it by $\Lk_p(\sigma, C)$. If $C$ is a geometric polytopal complex in $X^d=\R^d$ (or $X^d=S^d$), then $\Lk_p(\sigma, C)$ is naturally realized in $\RN^1_{(p,\sigma)} X^d$. Obviously, $\RN^1_{(p,\sigma)} X^d$ is isometric to a sphere of dimension $d-\dim \sigma -1$, and will be considered as such. Up to ambient isometry $\Lk_p(\sigma, C)$ and  $\RN^1_{(p,\sigma)} \tau$ in $ \RN^1_{(p,\sigma)} |C|$ or $\RN^1_{(p,\sigma)} X^d$ do not depend on $p$; for this reason, $p$ will be omitted in notation whenever possible. By convention, we define $\Lk(\emptyset, C)=C$.

If $C$ is simplicial and $v$ is a vertex of $C$, we have the combinatorial equivalence 
\[\Lk(v,C) = (C-v)\cap \St(v,C)=\St(v,C)-v.\] 

If $C$ is a simplicial complex, and $\sigma$, $\tau$ are faces of $C$, then $\sigma\ast \tau$ is the minimal face of $C$ containing both $\sigma$ and $\tau$ (assuming it exists). If $\sigma$ is a face of $C$, and $\tau$ is a face of $\Lk(\sigma,C)$, then $\sigma \ast \tau$ is the face of $C$ with $\Lk(\sigma,\sigma \ast \tau)=\tau$. In both cases, the operation~$\ast$ is called the \emph{join}.

%\enlargethispage{4mm}
\subsection{Collapsibility and non-evasiveness}
Inside a polytopal complex $C$, a \emph{free} face $\sigma$ is a face strictly contained in only one other face of $C$. An \emph{elementary collapse} is the deletion of a free face $\sigma$ from a polytopal complex~$C$. We say that $C$ \emph{(elementarily) collapses} onto $C-\sigma$, and write $C\searrow_e C-\sigma.$ We also say that the complex $C$ \emph{collapses} to a subcomplex $C'$, and write~$C\searrow C'$, if $C$ can be reduced to $C'$ by a sequence of elementary collapses. A \emph{collapsible} complex is a complex that collapses onto a single vertex. Collapsibility depends only on the face poset. 

%The connection to discrete Morse theory is highlighted by the following simple result:

%\begin{theorem}[Forman \cite{FormanADV}]\label{lem:equivalence}
%Let $C$ be a polytopal complex. The complex $C$ is collapsible if and only if $C$ admits a discrete Morse function with only one critical face.
%\end{theorem}

Collapsible complexes are contractible. Moreover, collapsible PL manifolds are necessarily balls \cite{Whitehead}.
Here are a few additional properties:

\begin{lemma}\label{lem:ccoll}
Let $C$ be a simplicial complex, and let $C'$ be a subcomplex of $C$.  Then the cone over base $C$ collapses to the cone over $C'$.
\end{lemma}

\begin{lemma}\label{lem:cecoll}
Let $v$ be any vertex of any simplicial complex $C$. If $\Lk(v,C)$ collapses to some subcomplex $S$, then $C$ collapses to 
$(C-v)\cup (v\ast S).$ In particular, if $\Lk(v,C)$ is collapsible, then  $C\searrow C-v$.  
\end{lemma}

\begin{lemma}\label{lem:uc}
Let $C$ denote a simplicial complex that collapses to a subcomplex $C'$. Let $D$ be a simplicial complex such that $D\cup C$ is a simplicial complex. If $D\cap C=C'$, then $D\cup C\searrow D$.
\end{lemma}

\begin{proof}
It is enough to consider the case $C\searrow_e C'=C-\sigma$, where $\sigma$ is a free face of $C$. The conclusion follows from the observation that the natural embedding $C\hookrightarrow D\cup C$ takes the free face $\sigma\in C$ to a free face of $D\cup C$.
\end{proof}

\emph{Non-evasiveness} is a further strengthening of collapsibility that emerged in theoretical computer science~\cite{KahnSaksSturtevant}. A $0$-dimensional complex is \emph{non-evasive} if and only if it is a point. Recursively, a $d$-dimensional simplicial complex ($d>0$) is \emph{non-evasive} if and only if there is some vertex $v$ of the complex whose link and deletion are both non-evasive. Again, non-evasiveness depends only on the face poset.

The notion of non-evasiveness is rather similar to vertex-decomposability, a notion defined only for \emph{pure} simplicial complexes \cite{ProvanBillera}; to avoid confusions, we recall the definition and explain the difference in the lines below. A $0$-dimensional complex is \emph{vertex-decomposable} if and only if it is a finite set of points. In particular, not all vertex-decomposable complexes are contractible. Recursively, a $d$-dimensional simplicial complex ($d>0$) is \emph{vertex-decomposable} if and only if it is pure and there is some vertex $v$ of the complex whose link and deletion are both vertex-decomposable (so in particular pure). All vertex-decomposable contractible complexes are non-evasive. % \cite{BTM}. ; the proof resembles the one similar to the basic fact that contractible shellable complexes are collapsible

An important difference arises when considering cones. It is easy to see that the cone over a simplicial complex $C$ is vertex-decomposable if and only if $C$ is. In contrast,

\begin{lemma}[cf.~Welker~\cite{Welker}] \label{lem:conev}
The cone over any simplicial complex is non-evasive.
\end{lemma}

By Lemma~\ref{lem:cecoll} every non-evasive complex is collapsible. As a partial converse, we also have the following lemma
\begin{lemma}[cf.~Welker~\cite{Welker}]
The derived subdivision of every collapsible complex is non-evasive. 
In particular, the derived subdivision of a non-evasive complex is non-evasive.
\end{lemma}

A \emph{non-evasiveness step} is the deletion from a simplicial complex $C$ of a single vertex whose link is non-evasive. Given two simplicial complexes $C$ and $C'$, we write \[C\searrow_{\NE} C' \] if there is a sequence of non-evasiveness steps that leads from $C$ to $C'$. We will need the following lemmas, which are well known and easy to prove: 
%, cf.~Welker~\cite{Welker}.

\begin{lemma}\label{lem:nonev}
If $C\searrow_{\NE} C'$, then $\mathrm{sd}^m  C   \searrow_{\NE}   \mathrm{sd}^m  C'$ for all non-negative $m$.
\end{lemma}

\begin{lemma}\label{lem:cone}
Let $v$ be any vertex of any simplicial complex $C$. Let $m \ge 0$ be an integer. Then 
\[ (\mathrm{sd}^m   C)-v   \searrow_{\NE}   \mathrm{sd}^m (C-v).\]
In particular, if $\mathrm{sd}^m  \Lk(v,C)$ is non-evasive, then $\mathrm{sd}^m   C   \searrow_{\NE}   \mathrm{sd}^m (C-v)$.
\end{lemma}

\begin{proof}
The case $m=0$ is trivial. We treat the case $m =1$ as follows: The vertices of $\sd C$ correspond to faces of~$C$, the vertices that have to be removed in order to deform $(\sd C) -v$ to $\sd (C-v)$ correspond to the faces of $C$ strictly containing $v$. The order in which we remove the vertices of $(\sd C)-v$ is by increasing dimension of the associated face. Let $\tau$ be a face of $C$ strictly containing $v$, and let $w$ denote the vertex of $\sd C$ corresponding to $\tau$. Assume all vertices corresponding to faces of $\tau$ have been removed from $(\sd C)-v$ already, and call the remaining complex~$D$. Denote by $\mathrm{L}(\tau,C)$ the set of faces of $C$ strictly containing $\tau$, and let $\F(\tau-v)$ denote the set of nonempty faces of $\tau-v$.
Then $\Lk(w, D)$ is combinatorially equivalent to the order complex of $\mathrm{L}(\tau,C)\cup \F(\tau-v)$, whose elements are ordered by inclusion. 
 Every maximal chain contains the face $\tau-v$, so $\Lk(w, D)$ is a cone, which is non-evasive by Lemma~\ref{lem:conev}. Thus, we have $D \searrow_{\NE} D-w$. The iteration of this procedure shows $(\sd C) -v \searrow_{\NE}  \sd (C-v)$, as desired.
The general case follows by induction on $m$: Assume that $m\geq 2$. Then 
\begin{align*}
(\mathrm{sd}^{m} C)-v\ =\  (\sd (\mathrm{sd}^{m-1} C))-v \  & \searrow_{\NE} \  \sd ((\mathrm{sd}^{m-1} C) - v)\\ & \searrow_{\NE}\  \sd (\mathrm{sd}^{m-1} (C - v))\ =\  \mathrm{sd}^m (C - v), 
\end{align*}
by applying the inductive assumption twice, and Lemma~\ref{lem:nonev} for the second deformation.
\end{proof}

\section{Non-evasiveness of star-shaped complexes}\label{ssc:stsh}
Here we show that \emph{any} subdivision of a star-shaped set becomes collapsible after $d-2$ derived subdivisions ({Theorem~\ref{thm:ConvexEndo}}); this proves Goodrick's conjecture up to a fixed number of subdivisions.% In the remaining part of this section, we work with geometric polytopal complexes. %Let us recall some definitions.

\begin{definition}[Star-shaped sets]
A subset $X \subset \R^d$ is \emph{star-shaped} if there exists a point $x$ in $X$, a \emph{star-center} of $X$, such that for each $y$ in $X$, the segment $[x,y]$ lies in $X$. 
Similarly, a subset $X \subset S^d$ is \emph{star-shaped} if $X$ lies in a closed hemisphere of $S^d$ and there exists a \emph{star-center} $x$ of $X$, in the interior of the hemisphere containing $X$, such that for each $y$ in $X$, the segment from $x$ to $y$ lies in $X$. With abuse of notation, a polytopal complex $C$ (in $\R^d$ or in $S^d$) is \emph{star-shaped} if its underlying space is star-shaped.  
\end{definition}

%Below we recall a classical notion in PL topology.

%\begin{definition}[Derived neighborhoods]
%Let $C$ be a polytopal complex. Let $D$ be a subcomplex of $C$. The \emph{(first) derived neighborhood} $N(D,C)$ of $D$ in $C$ is the simplicial complex
%\[N(D,C):=\bigcup_{\sigma\in \sd D} \St(\sigma,\sd C). \] 
%\end{definition}

%Any geometric realization of $\sd C$ naturally induces a corresponding geometric realization of $N(D,C)$. 
Via central projection, one can see that star-shaped complexes in $\R^d$ are precisely those complexes that have a star-shaped realization in the interior of a hemisphere of $S^d$. The more delicate situation in which $C$ is star-shaped in $S^d$ but touches the boundary of any closed hemisphere containing it, is addressed by the following lemma.

\begin{lemma}\label{lem:star-shaped}
Let $C$ be a star-shaped polytopal complex in a closed hemisphere $\overline{H}_+$ of $S^d$. Let $D$ be the subcomplex of faces of $C$ that lie in the interior of $ \overline{H}_+$. Assume that every nonempty face $\sigma$ of $C$ in $H:=\partial \overline{H}_+$ is the facet of a unique face $\tau$ of $C$ that intersects both $D$ and $H$. 
Then the complex $N(D,C)$ has a star-shaped geometric realization in $\R^d$.
\end{lemma}

\begin{proof}
Let $m$ be the midpoint of $\overline{H}_+$. Let $B_r (m)$ be the closed metric ball in $\overline{H}_+$ with midpoint $m$ and radius $r$ (with respect to the canonical space form metric $\mathrm{d}$ on $S^d$). If  $C\subset \intx \overline{H}_+$, then $C$ has a realization as a star-shaped set in $\R^d$ by central projection, and we are done. Thus, we can assume that $C$ intersects $H$. Let $x$ be a star-center of~$C$ in the interior of $\overline{H}_+$. Since $D$ and $\{x\}$ are compact and in the interior of $\overline{H}_+$, there is some real number $R < \nicefrac{\pi}{2}$ such that the ball $B_R(m)$ contains both $x$ and $D$. Let $J := (R, \nicefrac{\pi}{2}]$.

If $\sigma$ is any face of $C$ in $H$, let $v_\sigma$ be any point in the relative interior of $\sigma$. If $\tau$ is any face of $C$ intersecting~$D$ and $H$, define $\sigma(\tau):= \tau\cap H$. For each~$r\in J$, choose a point $w(\tau, r)$ in the relative interior of $B_r(m) \cap \tau$, so that for each $\tau$, $w(\tau, r)$ depends continuously on $r$ and tends to $v_{\sigma(\tau)}$ as $r$ approaches $\nicefrac{\pi}{2}$. Extend each family $w(\tau, r)$ continuously to $r=\nicefrac{\pi}{2}$ by defining $w(\tau, \nicefrac{\pi}{2}):=v_{\sigma(\tau)}$.

Next, we use these one-parameter families of points to produce a one-parameter family $N_r(D,C)$ of geometric realizations of $N(D,C)$, where $r\in (R, \nicefrac{\pi}{2}]$. For this, let $\varrho$ be any face of $C$ intersecting $D$. If $\varrho$ is in~$D$, let $x_{\varrho,r}$ be any point in the relative interior of $\varrho$ (independent of~$r$). If $\varrho$ is not in $D$, let $x_{\varrho,r}:=w(\varrho, r)$. We realize $N_r(D,C)\cong N(D,C)$, for $r$ in $J$, such that the vertex of $N_r(D,C)$ corresponding to the face $\varrho$ is~given the coordinates of $x_{\varrho,r}$. As the coordinates of the vertices determine a simplex entirely, this gives a realization $N_r(D,C)$ of $N(D,C)$, as desired.

To finish the proof, we claim that if $r$ is close enough to $\nicefrac{\pi}{2}$, then $N_r(D,C)$ is star-shaped with star-center $x$. Let us prove the claim. First we define the \emph{extremal faces} of $N_r(D,C)$ as the faces all of whose vertices are in~$\partial B_r(m)$. We say that a pair $\sigma,\ \sigma'$ of extremal faces is \emph{folded} if there are two points $a$ and $b$, in $\sigma$ and $\sigma'$ respectively, and satisfying $\mathrm{d}(a,b)< \pi$, such that the closed geodesic through $a$ and $b$ contains $x$, but the segment $[a, b]$ is not contained in~$N_r(D,C)$. 

When $r=\nicefrac{\pi}{2}$, folded faces do not exist since for every pair of points $a$ and $b$ in extremal faces, $\mathrm{d}(a,b)<\pi$, the geodesic through $a$ and $b$ lies in $\partial B_{\nicefrac{\pi}{2}}$; hence, all such subspaces have distance at least $\mathrm{d}(x,H)>0$ to $x$. 

Thus, since we chose the vertices of $N_r(D,C)$ to depend continuously on $r\in (R, \tfrac{\pi}{2}]$, we can find a real number $R'\in(R, \nicefrac{\pi}{2})$ such that for any $r$ in the open interval $J' := \left(R',\nicefrac{\pi}{2} \right),$ the simplicial complex $N_r(D,C)$ contains no folded pair of faces. 

But then, for every $r\in J'$, $N_r(D,C)$ folded pairs of extremal faces are avoided. Hence, for every $y$ in $N_r(D,C)$, the segment $[x,y]$ lies entirely in $N_r(D,C)$, since every part of the segment not in $N_r(D,C)$ must have as boundary points in a folded extremal pair.
\end{proof}

Recall that, if $(S, \prec)$ is an arbitrary poset and $S \subset T$, an \emph{extension} of $\prec$ to $T$ is any partial order $\widetilde{\prec}$ that coincides with $\prec$ when restricted to (pairs of elements in) $S$.
%such that $a\, \widetilde{\prec}\, b$ whenever $a \prec b$ for $a,b\in S$.

\begin{definition}[Derived order]\label{def:extord}
Let $C$ be a polytopal complex. Let $S$ denote a subset of mutually disjoint faces of $C$. We extend this order to an irreflexive partial order $\widetilde{\prec}$ on $C$ as follows: Let $\sigma$ be any face of $C$, and let $\tau\subsetneq \sigma$ be any strict face of $\sigma$. 
\begin{compactitem}[$\bullet$]
\item If $\tau$ is the minimal face of $\sigma$ under $\prec$, then $\tau\,\widetilde{\prec}\, \sigma$.
\item If $\tau$ is any other face of $\sigma$, then $\sigma\, \widetilde{\prec}\, \tau$.
\end{compactitem}
The transitive closure of the relation $\widetilde{\prec}$ gives an irreflexive partial order on the faces of~$C$. 

Note that by the correspondence of faces of $C$ to the vertices of $\sd  C$, it gives an irreflexive partial order on $\F_0(\sd  C)$. Any total order that extends the latter order is a \emph{derived order} of $\F_0(\sd  C)$ induced by $\prec$.
\end{definition}

\begin{definition}[$H$-splitting derived subdivisions]
Let $C$ be a polytopal complex in $\mathbb{R}^d$, and let $H$ be a hyperplane of $\mathbb{R}^d$. An \emph{$H$-splitting derived subdivision} of $C$ is a derived subdivision, with vertices chosen so that the following property holds: for any face $\tau$ of $C$ that intersects the hyperplane $H$ in the relative interior, the vertex of $\sd C$ that corresponds to $\tau$ in $C$ lies on the hyperplane $H$.
\end{definition}

\begin{figure}[htb]
\centering
\includegraphics[width=0.6\linewidth]{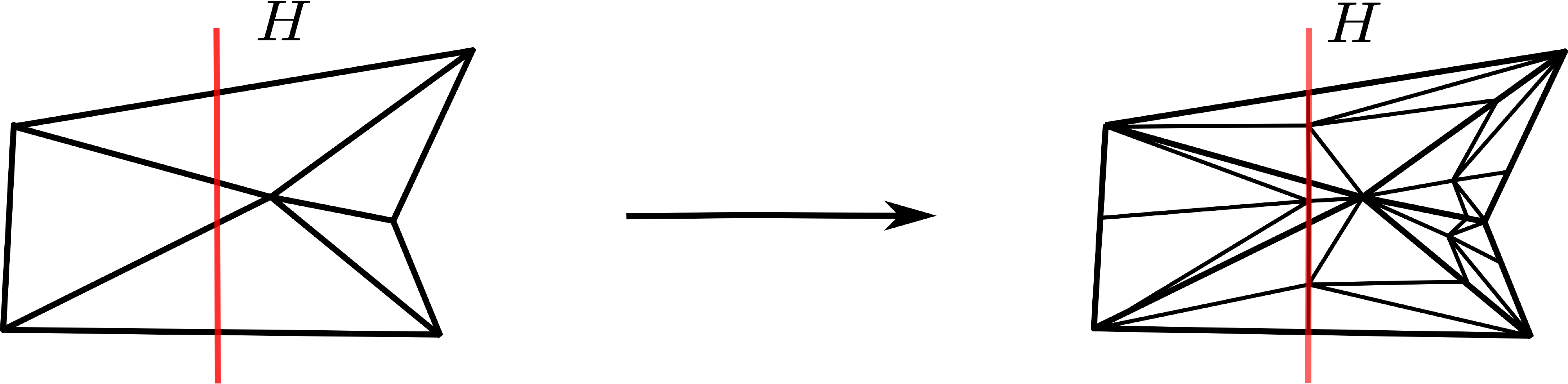}
\caption{An example of an $H$-splitting derived subdivision. (As we are free to choose where to position the new vertices, whenever possible we place them on the red hyperplane $H$.)}
\label{fig:split}
\end{figure}

\begin{definition}[Split link and Lower link]\label{def:slk}
Let $C$ be a simplicial complex in $\R^d$. Let $v$ be a vertex of $C$, and let $\overline{H}_+$ be a closed halfspace in $\R^d$ with outer normal $\nu$ at $v\in H:=\partial \overline{H}_+$. The \emph{split link} (of $C$ at $v$ with respect to $\nu$), denoted by $\SLk (v, C)$, is the intersection of $\Lk (v,  C)$ with the hemisphere $\TT_v^1 \overline{H}_+$, that is, \[\SLk (v, C):= \{\sigma \cap \TT_v^1 \overline{H}_+: \sigma \in \Lk(v,C)\}.\]
The \emph{lower link} $\LLk (v,C)$ of~$C$ at $v$ with respect to the direction $\nu$ is the restriction $\RS(\Lk(v,C), \intx \TT_v^1 \overline{H}_+)$ of $\Lk(v,C)$ to $\intx \TT_v^1 \overline{H}_+$. 

The complex $\LLk (v, C)$ is naturally a subcomplex of $\SLk (v,C)$: we have as an alternative definition of the lower link the identity \[\LLk (v,C)=\RS(\SLk(v,C), \intx \TT_v^1 \overline{H}_+).\]
\end{definition}

Finally, for a polytopal complex $C$ and a face $\tau$, we denote by $\mathrm{L}(\tau,C)$ the set of faces of $C$ strictly containing $\tau$.

\begin{theorem}\label{thm:ConvexEndo}
Let $C$ be a polytopal complex in $\R^d$, $d\geq 3$, or a simplicial complex in $\R^2$. If $C$ is star-shaped in $\R^d$, then $\mathrm{sd}^{d-2}  (C) $ is non-evasive, and in particular collapsible. 
\end{theorem}

\begin{proof}
The proof is by induction on the dimension. The case $d=2$ is easy: Every contractible planar simplicial complex is non-evasive \cite[Lemma 2.3]{AB-tight}. 

Assume now $d\ge 3$. Let $\nu$ be generic in $S^{d-1}\subset \R^{d}$, so that no edge of $C$ is orthogonal to~$\nu$. Let $H$ be a hyperplane through a star-center $x$ of $C$ such that $H$ is orthogonal to~$\nu$. Throughout this proof, let $\sd C$ denote any $H$-splitting derived subdivision of $C$. Let $\overline{H}_+$ (resp.~$\overline{H}_-$) be the closed halfspace bounded by $H$ in direction $\nu$ (resp.\ $-\nu$), and let $H_+$ (resp.\ $H_-$) denote their respective interiors. We make five claims: 
\begin{compactenum}[(1)]
\item For $v\in \F_0(\RS( C, {H}_+))$, the complex $\mathrm{sd}^{d-3}  N (\LLk (v, C),\Lk(v,C))$ is non-evasive. 
\item For $v\in \F_0(\RS( C, {H}_-))$, the complex $\mathrm{sd}^{d-3}  N (\mathrm{LLk}^{-\nu} (v, C), \Lk(v,C))$ is non-evasive. 
\item $\mathrm{sd}^{d-3}  \RS(\sd  C, \overline{H}_+) \; \searrow_{\NE} \; \mathrm{sd}^{d-3}  \RS( \sd C, H).$
\item $\mathrm{sd}^{d-3}  \RS(\sd  C, \overline{H}_-) \; \searrow_{\NE} \; \mathrm{sd}^{d-3}  \RS( \sd  C, H).$
\item $\mathrm{sd}^{d-3}  \RS( \sd  C, H)$ is non-evasive.
\end{compactenum}
%The combination of claims (3), (4) and (5) will imply the desired conclusion, as we will see.
Here are the respective proofs: 
\begin{compactenum}[(1)]
\item Let $v$ be a vertex of $C$ that lies in ${H}_+$. The complex $\SLk(v,C)$ is star-shaped in the $(d-1)$-sphere~$\TT^1_v \R^d$; its star-center is the tangent direction of the segment $[v,x]$ at $v$. Furthermore, $\nu$ is generic, so $\SLk(v,C)$ satisfies the assumptions of Lemma~\ref{lem:star-shaped}. The lemma tells us that the complex \[N(\LLk (v, C),\SLk(v,C))\cong N(\LLk (v, C),\Lk(v,C))\] has a star-shaped geometric realization in $\R^{d-1}$.  So by the inductive assumption, the simplicial complex $\mathrm{sd}^{d-3}  N (\LLk (v, C),\Lk(v,C))$ is non-evasive.

\item This is symmetric to (1).

\item Since $\nu$ is generic, the vertices of $C$ are totally ordered according to their value under the function $\langle \cdot, \nu  \rangle$. We adopt the convention that the first vertex is the one that maximizes the functional $\langle \cdot, \nu  \rangle$. 
Let us extend such order to a derived total order on the vertices of $\sd C$, as explained in Definition~\ref{def:extord}. 

Note that the order on the vertices of $\sd C$ \emph{does not} have to be induced by $\langle \cdot, \nu  \rangle$; it is however easy to arrange the vertices of $\sd C$ in such a way that both orders agree.

Let $v_0,v_1,\, \cdots, v_n$ be the vertices of $\RS(\sd  C, {H}_+)\subset \sd C$, labeled according to the derived order (so that the maximal vertex is $v_0$). Clearly, these vertices form an order filter (i.e.\ an upward closed subset) for the order on vertices of $\sd C$ defined above.

Define $C_i$ by restricting $\sd  C$ to $\overline{H}_+$, and deleting $\{v_0,\, \cdots, v_{i-1}\}$, i.e.\ 
\[C_i:=\RS(\sd  C, \overline{H}_+)-\{v_0,\, \cdots, v_{i-1}\}\] and 
define
\[\varSigma_i:=\mathrm{sd}^{d-3} C_i.\]
It remains to show that, for all $i$, $0\le i\le n$, we have $\varSigma_i\searrow_{\NE}\varSigma_{i+1}$. We distinguish two cases, according to whether $v_i$ was introduced when subdividing or not.
\begin{compactitem}[$\bullet$]
\item If $v_i$ corresponds to a face $\tau$ of $C$ of positive dimension, then let $w$ denote the vertex of $\tau$ minimizing~$\langle \cdot,\nu \rangle$. Following the definition of the derived order, the complex $\Lk(v_i, C_i)$ is combinatorially equivalent to the order complex associated to the set of faces $\mathrm{L}(\tau,C)\cup \{w\}$, whose elements are ordered by inclusion. 

Since $w$ is the unique minimum in that order, $\Lk(v_i, C_i)$ is combinatorially equivalent to a cone over base $\sd\Lk (\tau,C)$. But every cone is non-evasive (cf.~Lemma~\ref{lem:conev}). Thus, $C_i \searrow_{\NE} C_i - v_i=C_{i+1}$. By Lemma~\ref{lem:nonev}, $\varSigma_i \searrow_{\NE} \varSigma_{i+1}$.
\item If $v_i$ corresponds to an original vertex of $C$, we have by claim (1) that the $(d-3)$-rd derived subdivision of $\Lk (v_i,C_i)\cong N (\LLk (v_i, C),\Lk(v_i,C))$ is non-evasive. With Lemma~\ref{lem:cone}, we conclude that \[\varSigma_i = \mathrm{sd}^{d-3} C_i  \searrow_{\NE} \mathrm{sd}^{d-3} (C_i-v_i) = \mathrm{sd}^{d-3} C_{i+1} = \varSigma_{i+1}.\]
\end{compactitem}

Hence in both cases $\varSigma_i\searrow_{\NE}\varSigma_{i+1}$. This means that we can recursively delete one vertex, until the remaining complex has no vertex in ${H}_+$. Thus, the complex $\mathrm{sd}^{d-3} \RS(\sd C, \overline{H}_+)$ can be deformation retracted to $\mathrm{sd}^{d-3} \RS( \sd C, H)$ via non-evasiveness steps.
\item This is analogous to (3), exploiting claim (2) in place of claim (1).
\item This follows from the inductive assumption! In fact, $\RS(\sd C,H)$ is star-shaped in the $(d-1)$-dimensional hyperplane $H$: Hence, the inductive assumption gives that $\mathrm{sd}^{d-3}  \RS(\sd  C, H)$ is non-evasive.
\end{compactenum}
Once our five claims are established, we conclude by showing that claims (3), (4) and (5) imply that $\mathrm{sd}^{d-2}  C$ is non-evasive, as desired. First of all, we observe that if $A$, $B$ and $A\cup B$ are simplicial complexes with the property that $A , B \searrow_{\NE} A \cap B$, then $A\cup B \searrow_{\NE} A\cap B$. Applied to the complexes $\mathrm{sd}^{d-3}  \RS(\sd  C, \overline{H}_+)$, $\mathrm{sd}^{d-3}  \RS(\sd  C, \overline{H}_-)$ and $\mathrm{sd}^{d-2}  C=\mathrm{sd}^{d-3}  \RS(\sd  C, \overline{H}_+)\cup \mathrm{sd}^{d-3}  \RS(\sd  C, \overline{H}_-)$, the combination of (3) and (4) shows that $\mathrm{sd}^{d-2}  C   \searrow_{\NE} \mathrm{sd}^{d-3} \RS (\sd C,H) $, which in turn is non-evasive by (5).
\end{proof}

\section{Collapsibility of convex complexes}\label{ssc:convex}
As usual, we say that a polytopal complex $C$ (in $\R^d$ or in $S^d$) is \emph{convex} if its underlying space is convex. A hemisphere in $ S^d$ is in \emph{general position} with respect to a polytopal complex $C\in S^d$ if it contains no vertices of $C$ in the boundary. 

Since all convex complexes are star-shaped, the results of the previous section immediately imply that Lickorish's conjecture and Hudson's problem admit positive answer up to taking $d-2$ derived subdivisions. In a companion paper \cite{AB-Shellability}, we proved that one can do better: up to taking at most \emph{two} barycentric subdivisions, every convex complex is shellable, hence in particular collapsible.

In this section, we improve the previous results even further, by establishing that for Lickorish and Hudson's problems a positive answer can be reached after only \emph{one} derived subdivision (Theorems~\ref{thm:hudson} and~\ref{thm:liccon}). For this, we rely on the following Theorem~\ref{thm:ConvexEndo2}, the proof of which is very similar to the proof of Theorem~\ref{thm:ConvexEndo}.

\begin{theorem}\label{thm:ConvexEndo2}
Let $C$ be a convex polytopal $d$-complex in $S^d$ and let $\overline{H}_+$ be a closed hemisphere of $S^d$ in general position with respect to $C$. Then we have the following:
\begin{compactenum}[\rm (A)]
\item If $\partial C\cap \overline{H}_+=\emptyset$, then $N(\RS(C,\overline{H}_+),C)$ is collapsible.
\item If $\partial C\cap \overline{H}_+$ is nonempty, and $C$ does not lie in $\overline{H}_+$, then $N(\RS(C,\overline{H}_+),C)$ collapses to the subcomplex $N(\RS(\partial C,\overline{H}_+),\partial C)$.
\item If $C$ lies in $\overline{H}_+$, then there exists some facet $F$ of $\sd \partial C$ such that $\sd C$ collapses to $C_F:=\sd \partial C-F$.
\end{compactenum}
\end{theorem}

Part (C) of Theorem~\ref{thm:ConvexEndo2} can be equivalently rephrased as follows: ``If $C$ lies in $\overline{H}_+$, for \emph{any} facet $\sigma$ of $\sd \partial C$ the complex $\sd C$ collapses to $C_\sigma :=\sd \partial C- \sigma$''. In fact, for any $d$-dimensional simplicial  ball $B$, the following statements are equivalent \cite[Prp.\ 2.4]{BZ}: 
\begin{compactenum}[(i)]
\item $B\searrow \partial B- \sigma$ for some facet $\sigma$ of $\partial B$;
\item $B- F \searrow \partial B$ for any facet $F$ of $B$.
\end{compactenum}
\noindent Thus, we have the following corollary:

\begin{cor}\label{cor:ConvexEndo2}
Let $C$ be a convex polytopal complex in $\R^d$.  Then for any facet $\sigma$ of $C$ we have that $\sd C-\sigma$ collapses to $\sd \partial C$.
\end{cor}

\begin{proof}[\textbf{Proof of Theorem~\ref{thm:ConvexEndo2}}]
Claim (A), (B) and (C) can be proved analogously to the proof of Theorem~\ref{thm:ConvexEndo}, by induction on the dimension. Let us adopt convenient shortenings. We denote: 
        \begin{compactitem}[--]
\item by $\io_d$, the statement ``(A) is true for complexes of dimension $\le d$'';
\item by $\iit_d$, the claim ``(B) is true for complexes of dimension $\le d$''; and finally 
\item by $\iii_d$, the claim ``(C) is true for complexes of dimension $\le d$''. 
\end{compactitem}
Clearly $\io_0$, $\iit_0$, and $\iii_0$ are true. We assume, from now on, that $d>0$ and that $\io_{d-1}$, $\iit_{d-1}$ and $\iii_{d-1}$ are proven already. We then proceed to prove $\io_{d}$, $\iit_{d}$ and~$\iii_{d}$. Recall that $\mathrm{L}(\tau,C)$ denotes the faces of $C$ strictly containing a face $\tau$ of $C$. We will make use of the notions of derived order (Definition~\ref{def:extord}) and lower link (Definition~\ref{def:slk}) from the previous section.

\smallskip

\noindent \textbf{Proving} $\io_{d}$ \textbf{and} $\iit_{d}$: Let $x$ denote the midpoint  of $\overline{H}_+$, and let $\mathrm{d}(y)$ denote the distance of a point $y \in S^d$ to $x$ with respect to the canonical metric on $S^d$. 

If $\partial C\cap \overline{H}_+\neq \emptyset$, then $C$ is a polyhedron that intersects $S^d{{\setminus}}\overline{H}_+$ in its interior since $\overline{H}_+$ is in general position with respect to $C$. Thus, $\overline{H}_+{{\setminus}}C$ is star-shaped, and for every $p$ in $\intx C{{\setminus}}\overline{H}_+$, the point $-p\in \intx \overline{H}_+$ is a star-center for it. In particular, the set of star-centers of $\overline{H}_+{{\setminus}}C$ is open. Up to a generic projective transformation $\varphi$ of $S^d$ that takes $\overline{H}_+$ to itself, we shall assume that $x$ is a generic star-center of~$\overline{H}_+{{\setminus}} C$.

Let $\mathrm{M}(C, \overline{H}_+)$ denote the faces $\sigma$ of $\RS(C,\overline{H}_+)$ for which the function ${\min}_{y\in \sigma} \mathrm{d}(y)$ attains its minimum in the relative interior of $\sigma$. With this, we order the elements of $\mathrm{M}(C, \overline{H}_+)$ strictly by defining $\sigma\prec \sigma'$ whenever $\min_{y\in\sigma}\mathrm{d}(y)<\min_{y\in\sigma'}\mathrm{d}(y)$.

This allows us to induce an associated derived order on the vertices of $\sd C$, which we restrict to the vertices of $N(\RS(C,\overline{H}_+),C)$. Let $v_0, v_1, v_2, \, \cdots, v_n$ denote the vertices of $N(\RS(C,\overline{H}_+),C)$ labeled according to the latter order, starting with the maximal element $v_0$. Let $C_i$ denote the complex $N(\RS(C,\overline{H}_+),C)-\{v_0, v_1, \, \cdots, v_{i-1}\},$ and define 
\[\Sigma_i:=C_i\cup N(\RS(\partial C,\overline{H}_+),\partial C).\] We will prove that $\Sigma_i\searrow \Sigma_{i+1}$ for all $i$, $0\le i\le n-1$; from this $\io_{d}$ and $\iit_{d}$ will follow. There are four cases to consider here.

\begin{compactenum}[(1)]
\item $v_i$ is in the interior of $\sd C$ and corresponds to an element of $\mathrm{M}(C, \overline{H}_+)$.
\item $v_i$ is in the interior of $\sd C$ and corresponds to a face of $C$ not in $\mathrm{M}(C, \overline{H}_+)$.
\item $v_i$ is in the boundary of $\sd C$ and corresponds to an element of $\mathrm{M}(C, \overline{H}_+)$.
\item $v_i$ is in the boundary of $\sd C$ and corresponds to a face of $C$ not in $\mathrm{M}(C, \overline{H}_+)$.
\end{compactenum}
We need some notation to prove these four cases. Recall that we can define $\RN$, $\RN^1$ and $\Lk$ with respect to a basepoint; we shall need this notation in cases (1) and (3). Furthermore, let us denote by $\tau$ the face of $C$ corresponding to $v_i$ in $\sd C$, and let $m$ denote the point ${\arg\min}_{y\in \tau} \mathrm{d}(y)$. Finally, define the ball $B_m$ as the set of points $y$ in $S^d$ with~$\mathrm{d}(y)\le\mathrm{d}(m)$. 

\smallskip
\noindent \emph{Case $(1)$}:  The complex $\Lk(v_i,\varSigma_i)$ is combinatorially equivalent to $N(\mathrm{LLk}_m(\tau,C),\Lk_m(\tau,C))$, where \[\mathrm{LLk}_m(\tau,C):= \RS(\Lk_m(\tau,C),  \RN^1_{(m,\tau)} B_m)\] is the restriction of $\Lk_m(\tau,C)$ to the hemisphere $ \RN^1_{(m,\tau)} B_m$ of $\RN^1_{(m,\tau)} S^d$.  Since the projective transformation $\varphi$ was chosen to be generic, $\RN^1_{(m,\tau)} B_m$ is in general position with respect to $\Lk_m(\tau,C)$. Hence, by assumption $\io_{d-1}$, the complex \[N(\mathrm{LLk}_m(\tau,C),\Lk_m(\tau,C))\cong \Lk (v_i,\varSigma_i)\] is collapsible. Consequently,  Lemma~\ref{lem:cecoll} proves $\varSigma_i\searrow \varSigma_{i+1}=\varSigma_i-v_i.$

\smallskip
\noindent \emph{Case $(2)$}: If $\tau$ is not an element of $\mathrm{M}(C, \overline{H}_+)$, let $\sigma$ denote the face of $\tau$ containing $m$ in its relative interior. Then, $\Lk(v_i, \varSigma_i)=\Lk(v_i, C_i)$ is combinatorially equivalent to the order complex of the union $\mathrm{L}(\tau,C)\cup \sigma$, whose elements are ordered by inclusion. Since $\sigma$ is a unique global minimum of the poset, the complex $\Lk(v_i, \varSigma_i)$ is a cone, and in fact combinatorially equivalent to a cone over base $\sd \Lk(\tau, C)$. But all cones are non-evasive (Lemma~\ref{lem:ccoll}), so $\Lk(v_i, \varSigma_i)$ is collapsible. Consequently, Lemma~\ref{lem:cecoll} gives $\varSigma_i\searrow \varSigma_{i+1}=\varSigma_i-v_i$.

\smallskip 

%This takes care of $\io_{d}$. To prove case $\iit_{d}$, we have to consider the two additional cases in which $v_i$ is a boundary vertex of $\sd C$. 
%
\smallskip
\noindent \emph{Case $(3)$}: This time, $v_i$ is in the boundary of $\sd C$. As in case {(1)}, $\Lk (v_i,C_i)$ is combinatorially equivalent to the complex \[N(\mathrm{LLk}_m(\tau,C),\Lk_m(\tau,C)),\  \  \mathrm{LLk}_m(\tau,C):= \RS(\Lk_m(\tau,C),  \RN^1_{(m,\tau)} B_m)\] in the sphere $\RN^1_{(m,\tau)} S^d$. Recall that $\overline{H}_+{{\setminus}}C$ is star-shaped with star-center $x$ and that $\tau$ is not the face of $C$ that minimizes $\mathrm{d}(y)$ since $v_i\neq v_n$, so that $\RN^1_{(m,\tau)} B_m\cap \RN^1_{(m,\tau)} \partial C$ is nonempty. Since furthermore $\RN^1_{(m,\tau)} B_m$ is a hemisphere in general position with respect to the complex $\Lk_m(\tau,C)$ in the sphere $\RN^1_{(m,\tau)} S^d$, the inductive assumption $\iit_{d-1}$ applies: The complex $N(\mathrm{LLk}_m(\tau,C),\Lk_m(\tau,C))$ collapses to 
\[N(\mathrm{LLk}_m(\tau,\partial C),\Lk_m(\tau,\partial C))\cong \Lk (v_i,C'_i),\ \  C'_i:=C_{i+1} \cup (C_i\cap N(\RS(\partial C,\overline{H}_+),\partial C)).\]
Consequently, Lemma~\ref{lem:cecoll} proves that $C_i$ collapses to $C'_i$. Since \[ \varSigma_{i+1}\cap C_i =(C_{i+1}\cup N(\RS(\partial C,\overline{H}_+),\partial C))\cap C_i=  C_{i+1} \cup (C_i\cap N(\RS(\partial C,\overline{H}_+),\partial C))= C'_i\]
Lemma~\ref{lem:uc}, applied to the union $\varSigma_i=C_i\cup \varSigma_{i+1}$ of complexes $C_i$ and $\varSigma_{i+1}$ gives that $\varSigma_i$ collapses onto~$\varSigma_{i+1}.$

\smallskip
\noindent \emph{Case $(4)$}: As observed in case {(2)}, the complex $\Lk(v_i,C_i)$ is  combinatorially equivalent to a cone over base $\sd \Lk(\tau, C)$, which collapses to the cone over the subcomplex $ \sd  \Lk(\tau, \partial C)$ by Lemma~\ref{lem:ccoll}. Thus, the complex $C_i$ collapses to $C'_i:=C_{i+1}\cup(C_i \cap N(\RS(\partial C,\overline{H}_+),\partial C))$ by Lemma~\ref{lem:cecoll}. Now, we have $\varSigma_{i+1}\cap C_i=C'_i$ as in case {(3)}, so that $\varSigma_i$ collapses onto $\varSigma_{i+1}$ by Lemma~\ref{lem:uc}.

\smallskip

This finishes the proof of $\io_{d}$ and $\iit_{d}$ of Theorem~\ref{thm:ConvexEndo2}. It remains to prove the notationally simpler case $\iii_{d}$.

\smallskip

\noindent \textbf{Proving} $\iii_d$: Since $C$ is contained in the open hemisphere $\intx \overline{H}_+$, we may assume, by central projection, that $C$ is actually a convex simplicial complex in $\R^d$. Let $\nu$ be generic in $S^{d-1}\subset\R^d$.

The vertices of $C$ are totally ordered according to the decreasing value of $\langle \cdot, \nu \rangle$ on them. Let us extend this order to a total order on the vertices of $\sd C$, using the derived order. Let $v_i$ denote the $i$-th vertex of $\F_0(\sd C)$ in the derived order, starting with the maximal vertex $v_0$ and ending up with the minimal vertex $v_n$. 

The complex $\Lk(v_0, C)= \LLk(v_0,C)$ is a subdivision of the convex polytope $\TT^1_{v_0} C$ in the sphere $\TT^1_{v_0} \R^d$ of dimension $d-1$. By assumption $\iii_{d-1}$, the complex 
$
\Lk(v_0,\sd C)\cong \sd \Lk(v_0,C)
$
collapses onto $\partial \sd  \Lk(v_0, C)-F'$, where $F'$ is some facet of $\partial \Lk(v_0,\sd C)$. By Lemma~\ref{lem:cecoll}, the complex $\sd C$ collapses to \[\varSigma_1:=(\sd C-v_0)\cup (\partial \sd C -F) = (\sd C-v_0)\cup C_F,\] where $F:=v_0\ast F'$ and $C_F=\partial \sd C - F$. 

We proceed removing the vertices one by one, according to their position in the order we defined. More precisely, set $C_i:=\sd  C-\{v_0,\, \cdots, v_{i-1}\}$, and set $\varSigma_i:= C_i\cup C_F$. We shall now show that $\varSigma_i\searrow \varSigma_{i+1}$ for all $i$, $1\le i\le n-1$; this in particular implies $\iii_{d}$. There are four cases to consider:

\begin{compactenum}[(1)]
\item $v_i$ corresponds to an interior vertex of $C$.
\item $v_i$ is in the interior of $\sd C$ and corresponds to a face of $C$  of positive dimension.
\item $v_i$ corresponds to a boundary vertex of $C$.
\item $v_i$ is in the boundary of $\sd C$ and corresponds to a face of $C$ of positive dimension.
\end{compactenum}

\smallskip
\noindent \emph{Case $(1)$}: In this case, the complex $\Lk (v_i,\varSigma_i)$ is combinatorially equivalent to the simplicial complex $N(\LLk(v_i,C),\Lk(v_i,C))$ in the $(d-1)$-sphere $\TT^1_{v_i} \R^d$. By assumption $\io_{d-1}$, the complex \[N(\LLk(v_i,C),\Lk(v_i,C))\cong \Lk (v_i,\varSigma_i)\] is collapsible. Consequently, by Lemma~\ref{lem:cecoll}, the complex $\varSigma_i$ collapses onto $\varSigma_{i+1}=\varSigma_i-v_i$.

\smallskip
\noindent \emph{Case $(2)$}: If $v_i$ corresponds to a face $\tau$ of $C$ of positive dimension, let $w$ denote the vertex of $\tau$ minimizing~$\langle \cdot, \nu \rangle$. The complex $\Lk(v_i, \varSigma_i)$ is combinatorially equivalent to the order complex of the union $\mathrm{L}(\tau,C)\cup w$, whose elements are ordered by inclusion. Since $w$ is a unique global minimum of this poset, the complex $\Lk(v_i, \varSigma_i)$ is a cone (with a base naturally combinatorially equivalent to $\sd   \Lk(\tau, C)$). Thus, $\Lk(v_i, \varSigma_i)$ is collapsible since every cone is collapsible (Lemma~\ref{lem:ccoll}). Consequently, Lemma~\ref{lem:cecoll} gives $\varSigma_i\searrow \varSigma_{i+1}=\varSigma_i-v_i$.

\smallskip
\noindent \emph{Case $(3)$}: Similar to case {(1)}, $\Lk (v_i,C_i)$ is combinatorially equivalent to the derived neighborhood $N(\LLk(v_i,C),\Lk(v_i,C))$ of $\LLk(v_i,C)$ in $\Lk(v_i,C))$ in the $(d-1)$-sphere $\TT^1_{v_i} \R^d$. By assumption $\iit_{d-1}$, the complex \[N(\LLk(v_i,C),\Lk(v_i,C))\cong \Lk (v_i,C_i) \] collapses to 
\[N(\LLk(v_i, \partial C), \Lk(v_i,\partial C))\cong \Lk (v_i,C'_i),\ \  C'_i:=C_{i+1} \cup (C_i\cap C_F).\]
Consequently, Lemma~\ref{lem:cecoll} proves that $C_i$ collapses to $C'_i$. Now,
\[ \varSigma_{i+1}\cap C_i =(C_{i+1}\cup C_F)\cap C_i=  C_{i+1} \cup (C_i\cap C_F)= C'_i.\]
If we apply Lemma~\ref{lem:uc} to the union $\varSigma_i=C_i\cup \varSigma_{i+1}$, we obtain that $\varSigma_i$ collapses to the subcomplex~$\varSigma_{i+1}.$

\smallskip
\noindent \emph{Case $(4)$}: As in case {(2)} above, $\Lk(v_i,C_i)$ is naturally combinatorially equivalent to a cone over base $\sd   \Lk(\tau, C)$, which collapses to the cone over the subcomplex $ \sd  \Lk(\tau, \partial C)$ by Lemma~\ref{lem:ccoll}. Thus, the complex $C_i$ collapses to $C'_i:=C_{i+1}\cup(C_i \cap C_F)$ by Lemma~\ref{lem:cecoll}.

Now, we have $\varSigma_{i+1}\cap C_i=C'_i$ as in case {(3)}, so that $\varSigma_i$ collapses onto $\varSigma_{i+1}$ by Lemma~\ref{lem:uc}.
\end{proof}

\subsection*{Lickorish's conjecture and Hudson's problem}

In this section, we provide the announced partial answers to Lickorish's conjecture (Theorem~\ref{thm:liccon}) and Hudson's Problem (Theorem~\ref{thm:hudson}). 

\begin{theorem}\label{thm:hudson}
Let $C, C'$ be polytopal complexes such that $C' \subset C$ and $C\searrow C'$. Let $D$ denote any subdivision of $C$, and define $D':=\RS(D,C')$. Then, $\sd D   \searrow   \sd D'$.
\end{theorem}

\begin{proof}
It suffices to prove the claim for the case where %$C \searrow_e C'$; that is, we may assume that 
$C'$ is obtained from $C$ by a single elementary collapse; the claim then follows by induction on the number of elementary collapses. 
Let $\sigma$ denote the free face deleted in the collapsing, and let $\varSigma$ denote the unique face of $C$ that strictly contains it.

Let $(\delta, \varDelta)$ be any pair of faces of $\sd D$, such that $\delta$ is a facet of $\RS(\sd D,\sigma)$, $\varDelta$ is a facet of $\RS(\sd D,\varSigma)$, and $\delta$ is a codimension-one face of $\varDelta$. With this, the face $\delta$ is a free face of $\sd D$.
Now, by Corollary~\ref{cor:ConvexEndo2}, $\RS(\sd D,\varSigma)-\varDelta$ collapses onto $\RS(\sd D,\partial \varSigma)$. Thus, $\sd D-\varDelta$ collapses onto $\RS(\sd D,\sd D{{\setminus}}\rint\varSigma)$, or equivalently, \[\sd D-\delta\searrow \RS(\sd D,\sd D{{\setminus}}\rint\varSigma)-\delta.\] 
Now, $\RS(\sd D,\sigma)-\delta$ collapses onto $\RS(\sd D,\partial \sigma)$ by Corollary~\ref{cor:ConvexEndo2}, and thus \[\RS(\sd D,\sd D{{\setminus}}\rint\varSigma)-\delta \searrow \RS(\sd D,\sd D{{\setminus}}(\rint\varSigma\cup\rint\sigma)).\] To summarize, if $C$ can be collapsed onto $C-\sigma$, then 
\begin{align*}
\sd D  \ &\searrow_e \  \sd D-\delta \\ &\searrow \ \RS(\sd D,\sd D{{\setminus}}(\rint\varSigma\cup\rint\sigma)) =\  \RS(\sd D, C-\sigma)= \RS(\sd D, C').\qedhere 
\end{align*}
\end{proof}

\begin{lemma}[Bruggesser--Mani]\label{lem:BruMani} Let $\mu$ be any face of a $d$-dimensional polytope $P$. 
For any polytope $P$ and for any face $\mu$,  $P$ collapses onto 
$\St(\mu, \partial P)$, which is collapsible.
\end{lemma}

\begin{proof}
Performing a rocket shelling of $\partial P$ with a generic line through $\sigma$, one gets a shelling of $\partial P$ in which $\St(\sigma, \partial P)$ is shelled first~\cite[Cor.~8.13]{Z}. Let $\tau$ be the last facet of such a shelling. Any contractible shellable complex is collapsible; the collapsing sequence of the facets is given by the inverse shelling order. Therefore $\partial P - \tau$ collapses onto $\St(\sigma, \partial P)$. But as a polytopal complex, the complex $P$ collapses polyhedrally onto $\partial P - \tau$; so $P$ collapses also onto $\St(\sigma, \partial P)$. Being a cone, $\St(\sigma, \partial P)$ is collapsible.
\end{proof}

\begin{theorem}\label{thm:liccon}
Let $C$ denote any subdivision of a convex $d$-polytope. Then $\sd C$ is collapsible.
\end{theorem}

\begin{proof}
By Lemma~\ref{lem:BruMani}, the polytope $P$ is collapsible. Now for any subdivision $C$ of $P$, the facets of $C$ are all convex polytopes (not necessarily simplicial). Hence $\sd C$ is collapsible by Theorem~\ref{thm:hudson}.
\end{proof}

\section*{Acknowledgments}
Karim~Adiprasito acknowledges support by a Minerva fellowship of the Max Planck Society, an NSF Grant DMS 1128155, an ISF Grant 1050/16 and ERC StG 716424 - CASe.\\
Bruno~Benedetti acknowledges support by an NSF Grant 1600741, the DFG Collaborative Research Center TRR109, and the Swedish Research Council VR 2011-980. 
\\Part of this work was supported by the National Science Foundation under Grant No. DMS-1440140 while the authors were in residence at the Mathematical Sciences Research Institute in Berkeley, California, during the Fall 2017 semester.

{\small
\def\cprime{$'$}
\providecommand{\bysame}{\leavevmode\hbox to3em{\hrulefill}\thinspace}
\providecommand{\MR}{\relax\ifhmode\unskip\space\fi MR }
\providecommand{\MRhref}[2]{%
  \href{http://www.ams.org/mathscinet-getitem?mr=#1}{#2}
}
\providecommand{\href}[2]{#2}

}

\end{document}